\documentclass[11pt]{amsart}

\usepackage{amsfonts, amstext, amsmath, amsthm, amscd, amssymb}

\usepackage{graphicx, color}

\usepackage[margin=2.5cm]{geometry}

\usepackage{microtype}
\usepackage{tikz}
\usetikzlibrary{arrows.meta,calc,decorations.pathmorphing}
\usepackage{caption}

\usepackage[hidelinks,pagebackref,pdftex]{hyperref}

\renewcommand{\setminus}{{\smallsetminus}}

\newtheorem{theorem}{Theorem}[section]
\newtheorem{proposition}[theorem]{Proposition}
\newtheorem{lemma}[theorem]{Lemma}

\newtheorem*{namedtheorem}{\theoremname}
\newcommand{\theoremname}{testing}
\newenvironment{named}[1]{\renewcommand{\theoremname}{#1}\begin{namedtheorem}}{\end{namedtheorem}}

\theoremstyle{definition}
\newtheorem{definition}[theorem]{Definition}

\newtheorem{remark}[theorem]{Remark}

\title[Satellite Links with Full Twists and Single-Twist Companions]{Satellite Links with Multiple Full Twists and Single-Twist Companions}

\author{Thiago de Paiva}
\address{(Thiago de Paiva) Beijing International Center for Mathematical Research, Peking University, Beijing 100871, China P.R.}
\email{thhiagodepaiva@gmail.com}
\author{Yi Liu}
\address{(Yi Liu) Beijing International Center for Mathematical Research, Peking University, Beijing 100871, China P.R.}
\email{liuyi@bicmr.pku.edu.cn}
\author{Paolo Piccione}
\address{(Paolo Piccione) Institute of Mathematics and Statistics,
University of S\~ao Paulo, 05508-090, S\~ao Paulo--SP, Brazil}
\address{(Paolo Piccione) Department of Mathematics, School of Sciences,
Great Bay University, 523000, Dongguan--GD, China P.R.}
\email{paolo.piccione@usp.br}

\begin{document}
\maketitle

\begin{abstract}
We study the relationship between the number of full twists in positive braid representations of satellite links and their companion links. We construct infinitely many satellite links that admit positive braid representations with arbitrarily many full twists, while their companion links do not admit any positive braid representation with more than one full twist. This exhibits an unexpected divergence between the braid-theoretic complexity of a satellite link and that of its companion.
\end{abstract}

% Put title, author, abstract at the front
\maketitle

\section{Introduction}

Satellite links provide a fundamental construction in knot theory, allowing one
to build new links from existing ones. Informally, a satellite link is obtained
by taking a link $C$ (called the \emph{companion}) and replacing a tubular
neighborhood of each component by a link embedded in a solid torus (called the
\emph{pattern}); see Section~\ref{S1} for the formal definition. This
construction produces links whose topology reflects a combination of the
geometry of the companion and the structure of the pattern.

Understanding how properties of a satellite link relate to those of its companion is a natural and widely studied problem. In particular, it is often useful to compare geometric or algebraic features of a satellite link with those of its companion.

For instance, in recent work, the complexity of link complements---measured by the minimal number of tetrahedra in a triangulation---is compared between satellite knots and their companions; see, for example, \cite{Crushing}. This provides a geometric perspective on how the structure of a satellite link reflects that of its companion.

Motivated by these considerations, in this paper we study an analogous problem in the braid-theoretic setting, comparing the number of full twists in positive braid representatives of a satellite link and those of its companion.

%Let $B_p$ denote the braid group on $p$ strands, generated by the standard Artin generators $\sigma_1,\dots,\sigma_{p-1}$. A braid is said to be \emph{positive} if it can be expressed using only positive powers of these generators. We say that a braid $B \in B_p$ is a \emph{positive braid with $k$ full twists} if it contains $k$ positive full twists on all $p$ strands as a subword. Equivalently, there exist a positive braid $B_0 \in B_p$ and an integer $k \ge 1$ such that
%\[
%B = B_0(\sigma_1 \sigma_2 \cdots \sigma_{p-1})^{kp}.
%\]

Let $B_p$ denote the braid group on $p$ strands, generated by the standard Artin generators $\sigma_1,\dots,\sigma_{p-1}$. A braid is said to be \emph{positive} if it can be expressed using only positive powers of these generators.

Let
\[
\Delta^2_p = (\sigma_1 \sigma_2 \cdots \sigma_{p-1})^p
\]
denote the positive full twist on all $p$ strands.

A positive braid $\beta \in B_p$ is said to \emph{admit $k$ full twists} if there exists a positive braid $\beta_0 \in B_p$ such that
\[
\beta = \beta_0 (\Delta^2_p)^k.
\]

In the setting of braid theory, positive braid representations and the number of full twists provide a natural measure of complexity. This leads to the question of whether the braid-theoretic complexity of a satellite link constrains that of its companion.

In this paper, we study this question by comparing the number of full twists in positive braid representations of a satellite link and its companion. At first glance, one might expect that if a satellite link admits a positive braid representation with many full twists, then its companion should exhibit similar behavior.

Our main result shows that this is not the case.

\begin{theorem}\label{main}
There exist infinitely many satellite links $L$ admitting positive braid representations with arbitrarily many full twists, such that their companion links do not admit any positive braid representation with more than one full twist.
\end{theorem}

The examples arise from families of T-links, realized as satellites whose patterns are given by braids in the solid torus. Moreover, in our construction, the satellite structure is compatible with the braid representatives, in the sense that the companion tori can be chosen disjoint from the braid axis (see Section~\ref{S2}).

This phenomenon highlights an unexpected discrepancy between the braid-theoretic complexity of a satellite link and that of its companion.

%We note that results in the literature relate positive braid representations of satellite links and their companions (see \cite[Theorem 4.2]{Ito2025}). In our construction, the satellite links admit positive braid representatives with arbitrarily many full twists that are compatible with the underlying satellite structure. In contrast, the corresponding companion links do not admit positive braid representations with more than one full twist.

We note that \cite[Theorem~4.2]{Ito2025} relates positive braid representations of satellite links and their companions, suggesting a close relationship between their braid-theoretic complexity. Our main result shows that this relationship does not persist in general. In our construction, the satellite links admit positive braid representatives with arbitrarily many full twists that are compatible with the underlying satellite structure, while the corresponding companion links do not admit positive braid representations with more than one full twist.

This suggests that the relationship between the number of full twists of a satellite link and that of its companion is more subtle than previously understood. We therefore develop a refined analysis of the structure of the companion link, leading to a corrected formulation describing precisely how full twists may or may not be reflected from a satellite link to its companion  (see Section~\ref{S3}).

The paper is organized as follows. In Section~\ref{S1} we recall the necessary definitions on satellite links. In Section~\ref{S2} we prove Theorem~\ref{main}. In Section~\ref{S3} we analyze the braid structure of the companion link and obtain a refined description of the phenomenon.

\section{Satellite Links, Companions, and Patterns}\label{S1}

We briefly review the satellite construction for links and introduce the terminology that will be used throughout the paper. All links are assumed to be oriented.

Let $C = C_1 \cup \cdots \cup C_m \subset S^3$ be an $m$-component link, and let
\[
N(C) = \bigsqcup_{i=1}^m N(C_i)
\]
be a choice of disjoint tubular neighborhoods of its components.

The satellite construction replaces each neighborhood $N(C_i)$ with a solid torus containing a prescribed link. More precisely, for each $i=1,\dots,m$, a \emph{pattern component} is a pair $P_i=(V_i,\ell_i)$, where $V_i \cong S^1 \times D^2$ is a solid torus and $\ell_i \subset V_i$ is an oriented link not contained in any $3$-ball in $V_i$. The pattern component $P_i$ is said to be \emph{trivial} if $\ell_i$ is isotopic to the core curve $S^1 \times \{0\}$ of $V_i$.

A \emph{pattern} for the companion link $C$ is a collection $P=(P_1,\dots,P_m)$, and it is called \emph{nontrivial} if at least one of its components is nontrivial.

For each $i$, fix a homeomorphism
\[
f_i : V_i \longrightarrow N(C_i)
\]
that sends the preferred longitude of $V_i$ to the preferred longitude of $C_i$, namely a simple closed curve on $\partial N(C_i)$ having linking number zero with $C_i$, equivalently one that is null-homologous in $S^3 \setminus N(C_i)$. The resulting link
\[
C_P = \bigcup_{i=1}^m f_i(\ell_i) \subset S^3
\]
is called the \emph{satellite} of $C$ with pattern $P$.

We shall always regard the satellite presentation as part of the data. In
particular, the solid tori \(N(C_i)\) are chosen pairwise disjoint, and each
component of \(C_P\) is contained in exactly one of them. Equivalently, for
each \(i\), the solid torus \(N(C_i)\) contains precisely the components of
\(C_P\) arising from the pattern component \(P_i\), and no companion torus
\(f_j(\partial V_j)\), \(j\neq i\), is contained in \(N(C_i)\).

The images $f_i(\partial V_i)$ are called the \emph{companion tori}. The link $C_P$ is said to be a \emph{satellite link} if, for every nontrivial pattern component $P_i$, the corresponding torus $f_i(\partial V_i)$ is essential in the complement of $C_P$.

Thus, when we say that a companion torus $f_i(\partial V_i)$ is essential in
the complement of $C_P$, we mean that it is incompressible and not
boundary-parallel in
\[
S^3 \setminus \operatorname{int} N(C_P).
\]

A pattern component $P_i$ is called \emph{braided} if $\ell_i$ is a closed braid in $V_i$, that is, if it intersects each meridional disk transversely and positively. A pattern is braided if all its components are braided.

Finally, a link $L \subset S^3$ is called a \emph{satellite link} (respectively, a \emph{braided satellite link}) if it can be expressed as $L=C_P$ for some companion link $C$ and some nontrivial pattern $P$ (respectively, braided pattern $P$).

\begin{lemma}\label{lem:essential-tori}
Let \(L=C_P\) be a satellite link, and let
\[
T_i=f_i(\partial V_i)
\]
be one of the companion tori. Put \(W_i=f_i(V_i)\). Suppose that the pattern
component \(P_i=(V_i,\ell_i)\) is nontrivial. Suppose moreover that the outside
of \(T_i\), namely
\[
S^3\setminus \operatorname{int}(W_i),
\]
contains at least two components of \(L\). If \(L\) is non-split, then \(T_i\)
is essential in the complement of \(L\).
\end{lemma}

\begin{proof}
We show that \(T_i\) is incompressible and not boundary-parallel in
\(S^3\setminus L\).

First consider the inside of \(T_i\), namely \(W_i\setminus L\). Suppose that
\(T_i\) were compressible on the inside. Compressing \(T_i\) along such a disk
produces a sphere in \(W_i\) enclosing the pattern component \(f_i(\ell_i)\).
Hence \(f_i(\ell_i)\) would be contained in a \(3\)-ball in \(W_i\). Pulling
back by \(f_i\), this would imply that \(\ell_i\) is contained in a \(3\)-ball
in \(V_i\), contradicting the definition of satellite link. Thus
\(T_i\) is incompressible on the inside.

Now consider the outside of \(T_i\), namely
\[
S^3\setminus \operatorname{int}(W_i).
\]
Suppose that \(T_i\) were compressible on the outside. Compressing \(T_i\)
along an outside compressing disk produces a sphere separating the components
of \(L\) inside \(W_i\) from the components of \(L\) outside \(W_i\). Since both
sides contain components of \(L\), this sphere is a splitting sphere for \(L\).
This contradicts the assumption that \(L\) is non-split. Therefore \(T_i\) is
incompressible on the outside.

It remains to rule out boundary-parallelism. Suppose that \(T_i\) were
boundary-parallel in \(S^3\setminus L\). Since the outside of \(T_i\) contains
at least two components of \(L\), the torus \(T_i\) cannot be parallel through
the outside to the boundary of a tubular neighborhood of a single component of
\(L\). On the other hand, \(T_i\) cannot be boundary-parallel through the inside
either: such a parallelism would imply that the pattern component
\(f_i(\ell_i)\) is isotopic to the core of \(W_i\), contradicting the
nontriviality of \(P_i\).

Therefore \(T_i\) is incompressible and not boundary-parallel in
\(S^3\setminus L\). Hence \(T_i\) is essential.
\end{proof}

\section{Satellite Links with Full Twists and Single-Twist Companions}\label{S2}

In this section, we construct explicit families of satellite links whose braid-theoretic behavior differs significantly from that of their companion links. Our goal is to compare the number of full twists in positive braid representations of a satellite link and those of its companion. The examples we consider arise from families of T-links. We begin by identifying suitable satellite structures for these links and then analyze their braid representations in terms of full twists.

Given integers
\[
2 \le r_1 < \cdots < r_n \quad \text{and} \quad s_1,\dots,s_n > 0,
\]
the closure of the positive braid
\[
(\sigma_1 \cdots \sigma_{r_1-1})^{s_1}
(\sigma_1 \cdots \sigma_{r_2-1})^{s_2}
\cdots
(\sigma_1 \cdots \sigma_{r_n-1})^{s_n}
\]
defines a T-link, denoted
\[
T((r_1,s_1),\dots,(r_n,s_n)).
\]

\begin{definition}
Let \(p,q\) be positive integers with \(2\le p\le q\). Let
\[
2\le u_1<\cdots<u_m\le p
\quad\text{and}\quad
2\le r_1<\cdots<r_n<p
\]
be possibly empty sequences, and let
\[
v_1,\dots,v_m,s_1,\dots,s_n
\]
be positive integers. The corresponding \(V\)-braid is the braid on \(p\)
strands given by
\begin{align*}
&(\sigma_{p-1}\sigma_{p-2}\dots\sigma_{{p-u_1+1}})^{v_1}\dots(\sigma_{p-1}\sigma_{p-2}\dots\sigma_{{p-u_m+1}})^{v_m}(\sigma_1\sigma_2\dots\sigma_{r_1-1})^{s_1}\dots\\
&(\sigma_{1}\sigma_{2}\dots\sigma_{r_n-1})^{s_n}(\sigma_{1}\sigma_{2}\dots\sigma_{{p-1}})^{q}.
\end{align*}
where the corresponding product is omitted if one of the sequences is empty.
Its closure is called the \(V\)-link
\[
V((u_1,\overline{v_1}),\dots,(u_m,\overline{v_m}),
(r_1,s_1),\dots,(r_n,s_n),(p,q)).
\]
\end{definition}

Every V-link is represented by a positive braid with a full twist. Moreover, V-links correspond to the minimal braid representatives of T-links; see \cite[Theorem 3.9]{de2024lorenz}. 

\begin{proposition}\label{prop:T-links-nonsplit}
Every T-link is non-split.
\end{proposition}

\begin{proof}
Let \(L\) be a T-link. If \(L\) is a knot, then there is nothing to prove, so
assume that \(L\) has at least two components.

By \cite[Theorem~3.9]{de2024lorenz}, \(L\) admits a braid-index-realizing
representative which is a V-link braid. In particular, \(L\) is represented by
the closure of a positive braid containing at least one positive full twist.
Thus we may write this braid as
\[
\beta=\alpha\Delta_p^2,
\]
where \(\Delta_p^2\) is the positive full twist on \(p\) strands and \(\alpha\)
is a positive braid.

Recall that for an oriented link
\[
L=K_1\cup \cdots \cup K_\mu,
\]
the linking number \(\operatorname{lk}(K_i,K_j)\) is defined by
\[
\operatorname{lk}(K_i,K_j)=\frac12\sum_c \operatorname{sign}(c),
\]
where the sum is taken over all crossings between the components \(K_i\) and
\(K_j\) in any diagram of \(L\). This is a link invariant.

We claim that any two distinct components of \(\widehat{\beta}\) have positive
linking number. Let \(C\) and \(C'\) be distinct components of
\(\widehat{\beta}\). Suppose that \(C\) is supported on \(a\ge 1\) strands and
\(C'\) is supported on \(b\ge 1\) strands. Since \(\Delta_p^2\) is a full twist,
each strand of \(C\) crosses each strand of \(C'\) exactly twice in the full
twist, and all these crossings are positive. Hence the full twist contributes
\(2ab\) positive crossings between \(C\) and \(C'\), and therefore contributes
\[
ab
\]
to \(\operatorname{lk}(C,C')\). The remaining factor \(\alpha\) is positive, so
it can only add non-negative contributions to the linking number. Thus
\[
\operatorname{lk}(C,C')\ge ab\ge 1.
\]
Therefore every pair of distinct components of \(L\) has positive linking
number.

Now suppose, for contradiction, that \(L\) is split. Then there is a sphere in
\(S^3\setminus L\) separating the components of \(L\) into two non-empty
subsets. Choose components \(C\) and \(C'\) on opposite sides of this sphere.
Since they belong to different split summands, we have
\[
\operatorname{lk}(C,C')=0.
\]
This contradicts the positive linking number established above. Hence \(L\) is
non-split.
\end{proof}
 
\begin{lemma}\label{lemma3C}
Let $k$ be a non-negative integer. The T-link $$T((3, 1), (4+3k, 3))$$ has three link components.
\end{lemma} 
 
\begin{proof}
It follows from \cite[Theorem 1.1]{2025twisted}.
\end{proof} 

\begin{proposition}\label{pro1}
Let $a, b, c, k$ be positive integers, with at least one of $a,b,c$ greater than $1$. The T-link
\[
T((a+b+c, c), (a+b+2c+k(a+b+c), a+b+c))
\]
is a braided satellite link with companion
\[
T((3, 1), (4+3k, 3)).
\]
\end{proposition}

\begin{proof}
Let 
\[
L_1 = T((a+b+c, c), (a+b+2c+k(a+b+c), a+b+c)),
\]
which is the closure of the braid
\[
(\sigma_1 \cdots \sigma_{a+b+c-1})^{c}
(\sigma_1 \cdots \sigma_{a+b+2c+k(a+b+c)-1})^{a+b+c}.
\]
Equivalently, by conjugation, $L_1$ is the closure of the braid
\[
B_1 = (\sigma_1 \cdots \sigma_{a+b+2c+k(a+b+c)-1})^{a+b+c}
(\sigma_1 \cdots \sigma_{a+b+c-1})^{c}.
\]

We first show that the first $a+b$ strands of $B_1$ form parallel components in $L_1$. 
To see this, track these strands through the braid. After passing once through the first factor of $B_1$, they move to the strands indexed from 
\[
c+k(a+b+c)+1 \quad \text{to} \quad a+b+c+k(a+b+c)
\]
at the bottom of $B_1$, preserving their order from top to bottom, that is, each strand returns to the same relative position among these strands. After closing the braid, they return to the top and repeat this motion. Iterating this process $k$ times and passing through the first factor of $B_1$, they occupy the positions from $c+1$ to $a+b+c$ at the bottom of the first factor
\[
(\sigma_1 \cdots \sigma_{a+b+2c+k(a+b+c)-1})^{a+b+c}.
\]

Next, these strands pass through the second factor
\[
(\sigma_1 \cdots \sigma_{a+b+c-1})^{c},
\]
which involves only the first $a+b+c$ strands. After this, they occupy the first $a+b$ positions preserving their initial order. Then they return to their original positions under the braid closure. Hence these strands trace disjoint, parallel components in $L_1$.

Since these $a+b$ components are parallel, they lie in a solid torus and are isotopic to its core. Removing $a+b-2$ of these parallel components, we obtain the link
\[
L_2 = T((2+c, c), (2+2c+k(2+c), 2+c)),
\]
represented by the braid
\[
B_2 = (\sigma_1 \cdots \sigma_{2+2c+k(2+c)-1})^{2+c}
(\sigma_1 \cdots \sigma_{2+c-1})^{c}.
\]

We now analyze $L_2$. The components corresponding to strands $3$ through $2+c$ of $B_2$ are parallel. Indeed, tracking these strands through the braid shows that they move cyclically through the first factor of $B_2$ and, after $k+2$ passes through this factor, they occupy positions $1$ through $c$ preserving their initial order at the bottom of
\[
(\sigma_1 \cdots \sigma_{2+2c+k(2+c)-1})^{2+c}.
\]
Passing through the second factor
\[
(\sigma_1 \cdots \sigma_{2+c-1})^{c},
\]
they return to strands $3$ through $2+c$ preserving  their initial order and hence to their original positions after closure. Therefore, these $c$ components are parallel in $L_2$.

Removing $c-1$ of these parallel components yields the link
\[
L_3 = T((3, 1), (4+3k, 3)),
\]
which is represented by the braid
\[
B_3 = (\sigma_1 \cdots \sigma_{4+3k-1})^{3}(\sigma_1 \sigma_2).
\]

Let $C_1, C_2, C_3$ denote the components of $L_3$ corresponding to the first, second, and third strands of $B_3$, respectively. Reversing the above construction, we see that $L_1$ is obtained from $L_3$ by inserting $a$, $b$, and $c$ parallel components in tubular neighborhoods of $C_1$, $C_2$, and $C_3$, respectively. Each family of components is isotopic to the core of the corresponding solid torus. More precisely, each component of $L_1$ is contained in a solid torus associated to a component of the companion link, and within each such torus the pattern is realized as a closed braid.

It remains to verify that the companion tori associated to the nontrivial
patterns are essential in the complement of $L_1$. Since at least one of $a,b,c$ is greater than $1$, at least one of the patterns is nontrivial. Moreover, the link $L_1$ is a T-link, and hence is
non-split by Proposition~\ref{prop:T-links-nonsplit}. Each companion torus
separates the components of $L_1$ lying inside the corresponding tubular
neighborhood from the remaining components of $L_1$. In the present
construction, the outside of each companion torus contains at least two
components of $L_1$. Therefore, by Lemma~\ref{lem:essential-tori}, the
companion tori corresponding to the nontrivial patterns are
essential in the complement of $L_1$.

Therefore,
\[
T((a+b+c, c), (a+b+2c+k(a+b+c), a+b+c))
\]
is a braided satellite link with companion
\[
T((3, 1), (4+3k, 3)). \qedhere
\]
\end{proof}

Proposition~\ref{pro1} shows that the satellite structures appearing in this
setting are not isolated examples, but are part of the broader problem of
understanding the geometric types of T-links. These geometric structures have
been studied in several works, including results on torus, satellite, and
hyperbolic cases; see, for instance,
\cite{unexpected, de2021satellites, twofulltwists, de2022torus,
dePaivaPurcell2024, dePaiva2025, de2022hyperbolic}. We now introduce a new
explicit family of T-links admitting a natural braided satellite structure.

\begin{lemma}\label{lemC}
Let $k$ be a non-negative integer. The T-link
\[
T((3,1),(3+2k,2))
\]
has three components.
\end{lemma}

\begin{proof}
It follows from \cite[Theorem 1.1]{2025twisted}.
\end{proof}

\begin{proposition}\label{pro21}
Let $k$ be a non-negative integer. The T-link
\[
T((3,1),(4+3k,3))
\]
is a braided satellite link with companion
\[
T((3,1),(3+2k,2)).
\]
\end{proposition}

\begin{proof}
Let 
\[
L = T((3,1),(4+3k,3)),
\]
which is the closure of the braid
\[
B = (\sigma_1 \cdots \sigma_{4+3k-1})^{3}(\sigma_1 \sigma_2).
\]
The link $L$ has three components by Lemma~\ref{lemma3C}.

Let $L_3$ denote the component corresponding to the third strand of $B$. 
We describe how this strand moves through the braid. Starting at the third strand at the top, it passes through the braid $(\sigma_1 \cdots \sigma_{4+3k-1})^{3}$ and appears at the last strand (i.e., strand $4+3k$) at the bottom. After closing the braid, it returns to the same position at the top and repeats this process. Iterating this motion, after $k$ repetitions the strand reaches position $4$ at the top, and then moves down to position $1$ at the bottom of $(\sigma_1 \cdots \sigma_{4+3k-1})^{3}$. It then passes through the factor $(\sigma_1 \sigma_2)$, where it becomes the overstrand and ends at the third strand at the bottom of $B$. Finally, under the braid closure, it returns to its original position.

From this description, we see that the component $L_3$ is supported on a subset of $k+2$ strands. In particular, it contains the last overstrand of $(\sigma_1 \cdots \sigma_{4+3k-1})^{3}$ and the overstrand of $(\sigma_1 \sigma_2)$. Therefore, $L_3$ lies inside an embedded trivial solid torus $V$ whose boundary is a trivial torus $T$ intersecting the disc bounded by the braid axis once (see Figure~\ref{Fig2}).

\begin{figure}[htbp]
\centering
\begin{tikzpicture}[line cap=round,line join=round]

    % Elipse externa
    \draw[line width=1.2pt] (0,0) ellipse (4 and 1.5);

    % Pontos coloridos em RGB
    \foreach \x/\r/\g/\b in {
        -3/255/0/0,
        -2/0/150/0,
        -1/255/140/0,
         0/255/140/0,
         1/255/0/0,
         2/0/150/0,
         3/255/140/0
    }{
        \fill[fill={rgb,255:red,\r;green,\g;blue,\b}] (\x,0) circle (0.12);
    }

    % Curva laranja suave envolvendo apenas os pontos laranja
    \draw[line width=1.3pt, orange]
    plot[smooth cycle, tension=0.65] coordinates {
        (-1.60,0.30)
        (-1.00,0.75)
        (-0.6,0.95)
        (0.55,0.95)
        (1.45,0.75)
        (2.35,0.72)
        (3.30,0.55)
        (3.55,0.05)
        (3.35,-0.35)
        (2.85,-0.42)
        (2.45,-0.10)
        (2.10,0.28)
        (1.45,0.36)
        (0.85,0.32)
        (0.45,-0.12)
        (0.10,-0.58)
        (-0.65,-0.68)
        (-1.30,-0.45)
        (-1.60,-0.05)
    };

\end{tikzpicture}
\caption{Disk bounded by the braid axis with seven colored points representing the link components of $T((3,1),(7,3))$. The orange curve encloses exactly the three orange points and represents a meridian of a trivial solid torus containing the corresponding link component.}
\label{Fig2}
\end{figure}
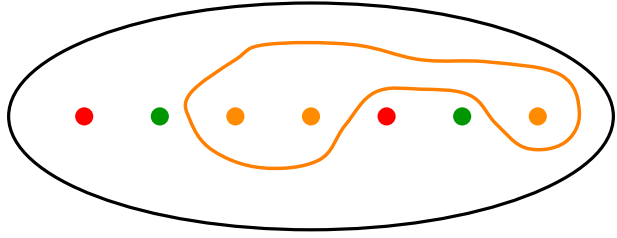

Now consider the link formed by the core of $V$ together with the other two components of $L$. This link is represented by the braid
\[
B_1 = (\sigma_1 \cdots \sigma_{3+2k-1})^{2}(\sigma_1 \sigma_2),
\]
which is precisely the T-link
\[
T((3,1),(3+2k,2)).
\]
The core of $V$ corresponds to the third strand of $B_1$.

It follows that $L$ is obtained from $T((3,1),(3+2k,2))$ by inserting braided patterns inside tubular neighborhoods of its components. More precisely, each component of $L$ lies inside a solid torus corresponding to a component of the companion link, and the pattern is given by closed braids in these solid tori.

It remains only to check essentiality of the companion tori associated to
nontrivial patterns. Since $L$ is a T-link, it is non-split by
Proposition~\ref{prop:T-links-nonsplit}. The construction above shows that
each such torus separates a nontrivial braided pattern from the remaining
components of $L$, with at least two components on the outside. Therefore
Lemma~\ref{lem:essential-tori} applies, and these tori are essential in the
complement of $L$.

Therefore,
\[
T((3,1),(4+3k,3))
\]
is a braided satellite link with companion
\[
T((3,1),(3+2k,2)). \qedhere
\]
\end{proof}

\begin{lemma}\label{lem1}
Let $a,b,c,k$ be positive integers, with at least one of $a,b,c$ greater than $1$. The link
\[
T((a+b+c, c), (a+b+2c+k(a+b+c), a+b+c))
\]
is a braided satellite link with companion
\[
T((3,1),(3+2k,2)).
\]
\end{lemma}

\begin{proof}
By Proposition~\ref{pro1}, the link
\[
L=T((a+b+c,c),(a+b+2c+k(a+b+c),a+b+c))
\]
is a braided satellite link with companion
\[
C_1=T((3,1),(4+3k,3)).
\]
That is, $L$ is obtained by inserting braided patterns inside tubular
neighborhoods of the components of $C_1$.

On the other hand, by Proposition~\ref{pro21}, the link $C_1$ is itself a
braided satellite link with companion
\[
C_0=T((3,1),(3+2k,2)).
\]
Thus $L$ is an iterated braided satellite link: first $C_1$ is obtained
from the companion $C_0$ by a braided satellite construction, and then
$L$ is obtained from $C_1$ by another braided satellite construction.

Equivalently, composing the two braided satellite structures gives a
braided satellite structure on $L$ whose companion is $C_0$. The essentiality of the resulting companion tori follows from the
essentiality established in Propositions~\ref{pro1} and~\ref{pro21}. Therefore
\[
T((a+b+c,c),(a+b+2c+k(a+b+c),a+b+c))
\]
is a braided satellite link with companion
\[
T((3,1),(3+2k,2)). \qedhere
\]
\end{proof}

\begin{lemma}\label{lem2}
Let $a,b,c,k$ be positive integers, with at least one of $a,b,c$ greater than $1$. The link
\[
L = T((a+b+c, c), (a+b+2c+k(a+b+c), a+b+c))
\]
is equivalent to the V-link
\[
V = V((a+b+c, \overline{c}), (a+b+c, (k+1)(a+b+c) + c)).
\]
In particular, $L$ admits a positive braid representative with $a+b+c$ strands and $(k+1)$ full twists. Moreover, the companion tori described in Lemma~\ref{lem1} can be isotoped to be disjoint from the braid axis of $V$.
\end{lemma}

\begin{proof}
The equivalence between $L$ and $V$ follows from \cite[Theorem~3.9]{de2024lorenz}, obtained by applying Proposition~2.1 and the isotopy described in Figure~5 of \cite{de2024lorenz}. In particular, this shows that $L$ admits a positive braid representative with $a+b+c$ strands and $(k+1)$ full twists.

By Lemma~\ref{lem1}, the link $L$ is a braided satellite link with companion
\[
C = T((3,1),(3+2k,2)).
\]
By Lemma~\ref{lemC}, the link $C$ has three components, and hence $\partial N(C)$ consists of three tori.

From the construction of $C$ in the proofs of Propositions~\ref{pro1} and~\ref{pro21}, two of these tori lie around collections of strands corresponding to the first $a+b$ strands of the braid representing $L$. In particular, these tori are disjoint from a circle encircling these strands, and intersect the corresponding spanning disk in two simple closed curves.

The isotopies used to transform $L$ into $V$ preserve this circle and its spanning disk. Therefore, after applying these isotopies, the corresponding two tori remain disjoint from the braid axis of $V$.

The third torus is carried to a torus that remains disjoint from the braid axis and intersects the disk bounded by the axis in a single simple closed curve, corresponding to a circle encircling the first $a+b$ strands of $V$, as illustrated in Figure~\ref{fig:Fig1}.

It follows that all companion tori can be isotoped to be disjoint from the braid axis of $V$, as claimed.
\end{proof}

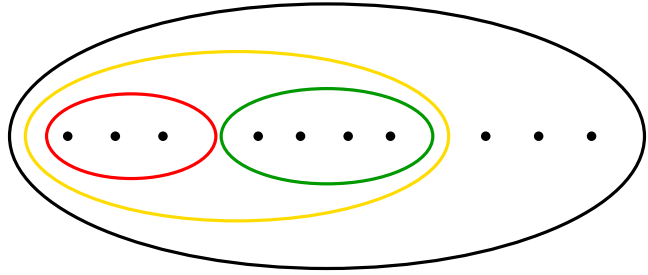
\begin{figure}[htbp] 
    \centering
    \begin{tikzpicture}[scale=0.7, line cap=round,line join=round]

        % Elipse externa
        \draw[line width=1.2pt] (-1,0) ellipse (6 and 2.5);

        % Elipse amarela (envolve ambas)
        \draw[line width=1.2pt, yellow!80!orange] (-2.7,0) ellipse (4 and 1.6);

        % Elipse vermelha
        \draw[line width=1.2pt, red] (-4.7,0) ellipse (1.6 and 0.8);

        % Elipse verde
        \draw[line width=1.2pt, green!60!black] (-1,0) ellipse (2.0 and 0.9);

        % Pontos na vermelha
        \fill (-5.9,0) circle (0.09);
        \fill (-5.0,0) circle (0.09);
        \fill (-4.1,0) circle (0.09);

        % Pontos na verde
        \fill (-2.3,0) circle (0.09);
        \fill (-1.5,0) circle (0.09);
        \fill (-0.6,0) circle (0.09);
        \fill (0.2,0) circle (0.09);

        % Pontos à direita
        \fill (2,0) circle (0.09);
        \fill (3,0) circle (0.09);
        \fill (4,0) circle (0.09);

    \end{tikzpicture}
    \caption{The disk bounded by the braid axis of the braid $V = V((10, \overline{3}), (10, (k+1)(10) + 3))$. The red, green, and yellow circles represent the intersections of this disk with the companion tori.}
    \label{fig:Fig1}
\end{figure}

\begin{lemma}\label{lem3}
Let $k$ be a non-negative integer. The link
\[
L = T((3,1),(3+2k,2))
\]
is equivalent to the V-link
\[
V((2,2k),(3,3)).
\]
In particular, $L$ admits a positive braid representation with a single full twist.
\end{lemma}

\begin{proof}
By \cite[Theorem 3.9]{de2024lorenz}, the link $L$ is equivalent to the V-link
\[
V((2,\overline{2k}),(3,3)).
\]
Applying the isotopy given by a $180^\circ$ rotation of the projection plane about a horizontal axis, as illustrated in \cite[Figure~5]{de2024lorenz}, transforms this diagram into
\[
V((2,2k),(3,3)).
\]
This proves the equivalence.

Since $V((2,2k),(3,3))$ is represented by a positive braid with a single full twist, the result follows.
\end{proof}

\begin{lemma}\label{lem4}
Let $a,b,c,k$ be positive integers, with at least one of $a,b,c$ greater than $1$. The link $$T((a+b+c, c), (a+b+2c+k(a+b+c), a+b+c))$$ is a satellite link with companion the V-link $$V((2, 2k), (3, 3)).$$
\end{lemma}

\begin{proof}
It follows from Lemmas~\ref{lem1} and \ref{lem3}. 
\end{proof}

\begin{proposition}\label{pro2}
Let $k$ be a non-negative integer. The link
\[
L = V((2,2k),(3,3))
\]
cannot be represented by any positive braid with more than one full twist.
\end{proposition}

\begin{proof} 
The link $L$ is represented as the closure of the braid
\[
(\sigma_1)^{2k}(\sigma_1\sigma_2)^3.
\]

Now observe that $(\sigma_1\sigma_2)^3=\Delta_3^2$ is the full twist on three strands, and its closure is the torus link $T(3,3)$, which has three components, each pair having linking number equal to $1$. Since $\sigma_1^{2k}$ is a pure braid, the closure of
\[
(\sigma_1)^{2k}(\sigma_1\sigma_2)^3
\]
also has three components. Moreover, each factor $\sigma_1^2$ contributes exactly $1$ to the linking number between the first and second components, and does not affect the linking numbers involving the third component. Hence the pairwise linking numbers of $L$ are
\[
\operatorname{lk}(K_1,K_2)=k+1,\qquad
\operatorname{lk}(K_1,K_3)=1,\qquad
\operatorname{lk}(K_2,K_3)=1.
\]
In particular, $L$ has a component whose linking number with each of the other two components is equal to $1$.

Suppose now that $L$ can be represented as the closure of a positive braid with more than one full twist. Then there exists a positive braid
\[
\beta=\alpha(\Delta_n^{2})^{r},\qquad r\ge 2,
\]
with $\alpha$ positive, such that $\widehat{\beta}\cong L$.

Let $C$ and $C'$ be two distinct components of $\widehat{\beta}$. Suppose that $C$ is obtained from $a\ge 1$ strands and $C'$ from $b\ge 1$ strands of the braid. Since $(\Delta_n^{2})^{r}$ is a pure braid, during the $r$ full twists each strand of $C$ crosses each strand of $C'$ exactly $2r$ times, all positively. Therefore the contribution of $(\Delta_n^{2})^{r}$ to the crossing number between $C$ and $C'$ is
\[
2r\,ab,
\]
so its contribution to the linking number is
\[
\frac12(2r\,ab)=rab.
\]
Because $r\ge 2$ and $a,b\ge 1$, we get
\[
\operatorname{lk}(C,C')\ge rab\ge 2.
\]
The remaining factor $\alpha$ is also positive, so it can only increase the linking number. Thus every pair of distinct components of the closure of $\beta$ must satisfy
\[
\operatorname{lk}(C,C')\ge 2.
\]

But this is impossible for $L$, since we already computed that $L$ has pairs of components with linking number equal to $1$. This contradiction shows that $L$ cannot be represented by any positive braid with more than one full twist.
\end{proof}

The previous results show that the links under consideration admit two complementary descriptions. On the one hand, they can be realized as satellite links with explicitly described companions. On the other hand, they admit positive braid representations whose number of full twists can be made arbitrarily large.

This leads naturally to the following question: does the existence of positive braid representations with many full twists for a satellite link impose similar constraints on its companion?

The next theorem shows that this is not the case.

\begin{named}{theorem~\ref{main}}
There exist infinitely many satellite links admitting positive braid representations with arbitrarily many full twists, whose companion links admit no positive braid representation with more than one full twist. Moreover, these braid representatives can be chosen so that the companion tori are disjoint from the braid axis.
\end{named}

\begin{proof}
By Lemma~\ref{lem4}, for positive integers $a,b,c,k$, with at least one of $a,b,c$ greater than $1$, the link
\[
L = T((a+b+c, c), (a+b+2c+k(a+b+c), a+b+c))
\]
is a braided satellite link with companion
\[
C = V((2,2k),(3,3)).
\]

By Lemma~\ref{lem2}, the link $L$ admits a positive braid representative $\beta$ on $a+b+c$ strands containing $(k+1)$ full twists. Since $k$ is arbitrary, this produces positive braid representations of $L$ with arbitrarily many full twists. Furthermore, by Lemma~\ref{lem2}, the companion tori $\partial N(C)$ are disjoint from the braid axis of $\beta$.

On the other hand, by Proposition~\ref{pro2}, the companion link
\[
C = V((2,2k),(3,3))
\]
does not admit any positive braid representation with more than one full twist.

Thus, for infinitely many choices of $k$, we obtain satellite links $L$ admitting positive braid representations with arbitrarily many full twists, while their companion links $C$ admit no positive braid representation with more than one full twist.
\end{proof}

\begin{figure}[htbp]
    \centering
    \begin{tikzpicture}[scale=0.7, line cap=round,line join=round]

        % --- REGIÃO (amarelo claro entre elipses) ---
        \fill[yellow!30, even odd rule]
            (-1,0) ellipse (6 and 2.5)
            (-2.7,0) ellipse (4 and 1.6);

        % --- PREENCHIMENTOS INTERNOS ---
        % Região vermelha (vermelho claro)
        \fill[red!25] (-4.7,0) ellipse (1.6 and 0.8);

        % Região verde (verde claro)
        \fill[green!25] (-1,0) ellipse (2.0 and 0.9);

        % --- CONTORNOS ---
        % Elipse externa
        \draw[line width=1.2pt] (-1,0) ellipse (6 and 2.5);

        % Elipse amarela (contorno mais escuro)
        \draw[line width=1.4pt, yellow!80!orange] (-2.7,0) ellipse (4 and 1.6);

        % Elipse vermelha
        \draw[line width=1.2pt, red] (-4.7,0) ellipse (1.6 and 0.8);

        % Elipse verde
        \draw[line width=1.2pt, green!60!black] (-1,0) ellipse (2.0 and 0.9);

        % --- PONTOS ---
        % Pontos na vermelha
        \fill (-5.9,0) circle (0.09);
        \fill (-5.0,0) circle (0.09);
        \fill (-4.1,0) circle (0.09);

        % Pontos na verde
        \fill (-2.3,0) circle (0.09);
        \fill (-1.5,0) circle (0.09);
        \fill (-0.6,0) circle (0.09);
        \fill (0.2,0) circle (0.09);

        % Pontos à direita
        \fill (2,0) circle (0.09);
        \fill (3,0) circle (0.09);
        \fill (4,0) circle (0.09);

    \end{tikzpicture}
    \caption{The shaded red and green disks represent meridional disks of two distinct companion solid tori, while the shaded yellow annulus corresponds to a portion of a meridional disk of the third companion solid torus.}
    \label{fig:Fig2}
\end{figure}
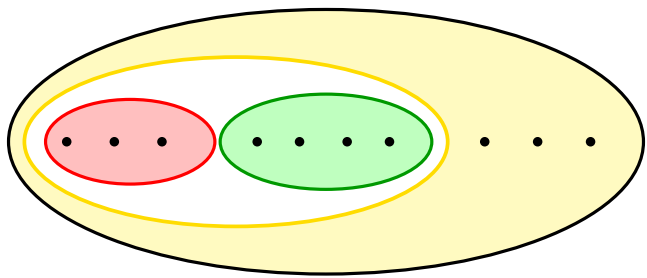

The preceding theorem reveals a clear discrepancy between the braid structure
of a satellite link and that of its companion. In the families constructed
here, the satellite links admit positive braid representatives with
arbitrarily many full twists that are compatible with the satellite
structure, whereas the companion links are restricted to positive braid
representations with at most one full twist.

We refer to \cite[Definition~3.3]{Ito2025} for the notion of compatibility of
a braid representative $\beta$ with the companion tori $\partial N(C)$. By
\cite[Lemma 3.1]{Ito2025}, this is equivalent to requiring that
$\partial N(C)$ be disjoint from the braid axis of $\beta$.

Notice that, in the present setting, the compatibility condition does not
need to be imposed as an additional hypothesis. Indeed, by
\cite[Theorem 5.3]{Ito2025}, if a satellite link is represented by a positive
braid containing a full twist, then the companion tori can be isotoped to be
disjoint from the braid axis.

This phenomenon can be explained geometrically by the relative position of the companion tori with respect to the braid axis. In our construction, one of the companion tori bounds a solid torus that separates the braid axis from the remaining ones. More precisely, this torus decomposes the ambient space into two solid tori: one containing the braid axis and the other containing the remaining components of the companion link, as illustrated in Figure~\ref{fig:Fig2}. When full twists are applied about the braid axis, they induce twisting only in the region containing the braid axis, affecting the components of the link contained in that solid torus. In contrast, the core of the complementary solid torus, which corresponds to a component of the companion link, remains unaffected by this twisting. As a result, the satellite link may acquire arbitrarily many full twists, while the corresponding companion component retains a rigid braid structure.

This shows that, even in the presence of compatible braid representatives, the number of full twists in a satellite link does not necessarily transfer to its companion.

This behavior contrasts with results in the literature that relate positive braid representatives of satellite links and their companions (see \cite[Theorem~4.2]{Ito2025}), and indicates that a more refined description of the companion is required.

Such a refinement is developed in the next section, where we analyze the braid structure of the companion link in greater detail.

\section{Braid Structure of the Companion}\label{S3}

The examples constructed in the previous section show that a satellite link may admit positive braid representatives with arbitrarily many full twists, while its companion link is severely restricted in this regard. This naturally raises the question of how the companion link is reflected in such braid representations.

In this section, we analyze the braid structure of the companion link in greater detail. Our goal is to understand how the presence of multiple full twists in a positive braid representative of a satellite link constrains the possible braid representations of its companion.

We show that, under these conditions, the companion link falls into one of two distinct types. In particular, either the companion inherits the full twists of the satellite link, or it admits a specific constrained form as a positive braid with a single full twist. This dichotomy provides a more precise description of the relationship between the braid structures of a satellite link and its companion.

Let $B \in B_p$ be a braid on $p$ strands. We say that the \emph{span} of $B$, denoted $\mathrm{span}(B)$, is contained in $[1,q]$ if all crossings of $B$ occur among strands indexed by $\{1,\dots,q\}$, where $q \le p$. Equivalently, $B$ can be written using only the generators $\sigma_1,\dots,\sigma_{q-1}$ and their inverses.

\begin{theorem}\label{theorem2}
Let $L$ be a satellite link with companion $C$ represented by a positive braid $\beta$ with $k$ full twists, and let $A$ be its braid axis. Let
\[
T_1,\dots,T_n \subset \partial N(C)
\]
be the companion tori associated to the companion link $C$, chosen disjoint from $A$.

Then the braid of $C$ is determined by the position of the companion tori relative to the braid axis as follows:

\begin{enumerate}
    \item If, for every companion torus $T_i$ which separates $S^3$ into two solid tori, all the other companion tori lie on the side containing the braid axis $A$, then $C$ is represented by a positive braid with $k$ full twists and with the same braid axis $A$ as $\beta$;

    \item If there exists a companion torus $T_i$ which separates $S^3$ into two solid tori such that all the other companion tori lie on the side not containing the braid axis $A$, then $C$ is represented by a positive braid with one full twist of the form
    \[
    B_0(\sigma_1 \cdots \sigma_{a-2})^{(k-1)(a-1)}
    (\sigma_1 \cdots \sigma_{a-1})^{a},
    \]
    where
    \[
    B_0(\sigma_1 \cdots \sigma_{a-2})^{(k-1)(a-1)} \in B_a
    \]
    and
    \[
    \mathrm{span}\!\left(
    B_0(\sigma_1 \cdots \sigma_{a-2})^{(k-1)(a-1)}
    \right)\subset [1,a-1].
    \]
\end{enumerate}
\end{theorem}

\begin{proof}
Let $A$ be the braid axis of the given positive braid representative $\beta$ of $L$, and let
\[
T_1,\dots,T_n \subset \partial N(C)
\]
be the family of companion tori associated to $C$, where $n \ge 1$.

By \cite[Theorem 5.3 and Lemma 3.1]{Ito2025}, after isotopy we may assume that each $T_i$ is disjoint from the braid axis $A$.

By the definition of a satellite link, for each $T_i$, all other tori in $T_1,\dots,T_n$ lie entirely in one of the two components of $S^3 \setminus T_i$.

Among the tori $T_1,\dots,T_n$, let
\[
T'_1,\dots,T'_j
\]
be those which separate $S^3$ into two solid tori. For each $T'_i$, write
\[
S^3 \setminus T'_i = V^i_1 \cup V^i_2,
\]
where $V^i_2$ is the solid torus containing the braid axis $A$. %Since the tori in $\partial N(C)$ are nested and define the satellite structure of $L$, for each such $T'_i$ all the other companion tori lie entirely in one of the two solid tori bounded by $T'_i$.

We now distinguish two cases.

\medskip
\noindent
\textbf{Case 1.} For every $T'_i$, all the remaining companion tori lie in $V^i_2$.

In this case every companion torus is nested on the same side of the braid axis. Therefore the argument of \cite[Theorem 4.2]{Ito2025} applies directly: taking the cores of the tori
\[
T_1,\dots,T_n
\]
produces a link isotopic to the companion $C$, and this link is represented by a positive braid with $k$ full twists and braid axis $A$. This gives conclusion~(1).

\medskip
\noindent
\textbf{Case 2.} There exists one $T'_i$ such that all the remaining companion tori lie in $V^i_1$.

Remove from $L$ all link components contained in $V^i_2$, and remove the torus $T'_i$ from the family
\[
T_1,\dots,T_n.
\]
Denote by $L'$ the resulting satellite link, and let $C'$ be the link given by the cores of the remaining tori. By construction, $L'$ is still represented by a positive braid with $k$ full twists.

Moreover, the choice of $T'_i$ implies that for every companion torus of $L'$, all the other companion tori lie on the side containing the braid axis. Hence Case~1 applies to $L'$. It follows that $C'$ is represented by a positive braid, say $\beta''$, with $k$ full twists.

Now view $C'$ inside the solid torus $V^i_1$. Since $T'_i=\partial V^i_1$, and $C'$ is braided in $V^i_1$, the torus $T'_i$ is boundary parallel to the boundary of a tubular neighbourhood of the braid axis of the braid representative $\beta''$ of $C'$. Equivalently, the original companion $C$ is obtained by adjoining to $C'$ the core of the solid torus bounded by $T'_i$ that contains the braid axis $A$, that is, the braid axis of $\beta''$.

Thus $C$ is the link formed by $\beta''$ together with its braid axis. Hence $C$ admits a positive braid representative of the form
\[
B_0(\sigma_1 \cdots \sigma_{a-2})^{(k-1)(a-1)}(\sigma_1 \cdots \sigma_{a-1})^{a},
\]
where
\[
B_0(\sigma_1 \cdots \sigma_{a-2})^{(k-1)(a-1)} \in B_a
\]
and
\[
\mathrm{span}\!\left(B_0(\sigma_1 \cdots \sigma_{a-2})^{(k-1)(a-1)}\right)\subset [1,a-1].
\]
This is exactly conclusion~(2).

Therefore one of the two alternatives must hold.
\end{proof}

\begin{remark}
We first observe that the conclusion of
\cite[Theorem~1.1]{Ito2025} does not hold in general without
additional geometric hypotheses.

Consider the link
\[
L=T((6,2),(8,6)),
\]
given as the closure of the braid
\[
\beta=
(\sigma_1\sigma_2\sigma_3\sigma_4\sigma_5\sigma_6\sigma_7)^6
(\sigma_1\sigma_2\sigma_3\sigma_4\sigma_5)^2.
\]

By Lemma~\ref{lem2}, taking $a=b=c=2$ and $k=0$, the link $L$
admits a positive braid representative with one full twist.

Moreover, by Proposition~\ref{pro1}, the link $L$ admits a
satellite structure with companion
\[
T((3,1),(4,3)).
\]
By Proposition~\ref{pro21}, this companion is itself a braided
satellite link with companion
\[
T((3,1),(3,2)).
\]

In the corresponding satellite structure, one of the companion tori,
say $T_3$, encloses the strands $5,6,7,8$ of $\beta$. The core of
this torus is the unknot, represented by the trivial one-strand braid
(see the proofs of Propositions~\ref{pro1} and~\ref{pro21}). The
link contained in $T_3$ is given by
\[
P_3=(\sigma_1\sigma_2\sigma_3)^2(\sigma_1)^2.
\]

Since the core of $T_3$ is the unknot, represented by the trivial
one-strand braid, the link $P_3$ defines a pattern associated to the
corresponding component $C_3$ of the companion link. Indeed,
identifying the solid torus bounded by $T_3$ with a standard solid
torus sends the preferred longitude to the preferred longitude of
$C_3$, and in this case these longitudes coincide up to isotopy.

Since the core of $T_3$ is the unknot, we have
\[
g(C_3)=0
\qquad\text{and}\qquad
b(C_3)=1,
\]
so that
\[
2g(C_3)+2b(C_3)-1=1.
\]

According to \cite[Theorem~1.1]{Ito2025}, the pattern $P_3$ should
contain at least one full twist. However, $P_3$ admits no positive
braid representative with a full twist, showing that the conclusion
of \cite[Theorem~1.1]{Ito2025} does not hold in this generality.

We conclude that the conclusion of \cite[Theorem~1.1]{Ito2025} holds under additional assumptions, namely those described in the first item of Theorem~\ref{theorem2}. More precisely, let $L$ be a satellite link represented by a positive braid with $k$ full twists and braid axis $A$, and let \[ T_1,\dots,T_n \subset \partial N(C) \] be the companion tori associated to the companion link $C$, chosen disjoint from $A$. Suppose that for every companion torus $T'_i$ that separates $S^3$ into two solid tori, all the remaining companion tori lie in the component that contains the braid axis $A$. Then the conclusion of \cite[Theorem~1.1]{Ito2025}, namely that the pattern must contain at least $2g(C_i)+2b(C_i)-1$ full twists, holds.
\end{remark}

\bibliographystyle{amsplain}  %% Uses AMS format for bibliography

%% Put all the bib entries in a file references.bib
\bibliography{Satellite}

\end{document}